\documentclass[11pt]{amsart}

\usepackage{amsmath}
\usepackage{amsfonts}
\usepackage{amssymb}
\usepackage{amscd}
\usepackage{amsthm}
\usepackage{filecontents}
\usepackage{framed}
\usepackage{fullpage}
\usepackage{latexsym}

\usepackage{hyperref}
\hypersetup{colorlinks,linkcolor={red},citecolor={blue},urlcolor={blue}}

\evensidemargin\oddsidemargin

\newtheorem{theorem}{Theorem}[section]

\newtheorem{corollary}[theorem]{Corollary}
\newtheorem{proposition}[theorem]{Proposition}

\theoremstyle{definition}
\newtheorem{definition}[theorem]{Definition}

\newtheorem{remark}[theorem]{Remark}

\numberwithin{equation}{section}

\newcommand{\CC}{\mathbb C}
\newcommand{\HH}{\mathbb H}
\newcommand{\NN}{\mathbb N}

\newcommand{\ZZ}{\mathbb Z}

\newcommand{\cD}{\mathcal D}
\newcommand{\cE}{\mathcal E}
\newcommand{\cH}{\mathcal H}
\newcommand{\SL}{\mathop{\mathrm {SL}}\nolimits}

\newcommand{\SO}{\mathop{\mathrm {SO}}\nolimits}

\newcommand{\Orth}{\mathop{\null\mathrm {O}}\nolimits}
\newcommand{\rank}{\mathop{\mathrm {rank}}\nolimits}
\newcommand{\latt}[1]{{\langle{#1}\rangle}}

\newcommand{\Grit}{\operatorname{Grit}}
\newcommand{\II}{\operatorname{II}}
\newcommand{\Borch}{\operatorname{Borch}}
\newcommand{\im}{\operatorname{Im}}
\newcommand{\norm}{\operatorname{Norm}}
\newcommand{\Proj}{\operatorname{Proj}}

\def\div{\operatorname{div}}
\newcommand{\m}{\operatorname{mod}}

\usepackage{longtable}

\usepackage{tikz}
\usetikzlibrary{positioning}
\usetikzlibrary{arrows}

\begin{document}

\title[On some free algebras of orthogonal modular forms II]{On some free algebras of orthogonal modular forms II}

\author{Haowu Wang}

\address{Max-Planck-Institut f\"{u}r Mathematik, Vivatsgasse 7, 53111 Bonn, Germany}

\email{haowu.wangmath@gmail.com}


\date{\today}


\begin{abstract}
In this paper we construct 16 free algebras of modular forms on type IV symmetric domains for some reflection groups related to the eight rescaled root lattices $A_1(2)$, $A_1(3)$, $A_1(4)$, $2A_1(2)$, $A_2(2)$, $A_2(3)$, $A_3(2)$, $D_4(2)$. 
As a corollary, we prove the modularity of formal Fourier--Jacobi expansions for these reflection groups.
\end{abstract}

\maketitle

\section{Introduction}
It is a classical problem to determine the algebra of automorphic forms on a symmetric domain. We are interested in the case of orthogonal modular forms, i.e. automorphic forms on symmetric domains of type IV for orthogonal groups of signature $(2,n)$.

Let $M$ be an even lattice of signature $(2,n)$ with $n\geq 3$. We define the affine cone $\mathcal{A}(M)$ as one of the two complex conjugate connected components of the space
$$
\{ \mathcal{Z} \in M\otimes\CC :  (\mathcal{Z}, \mathcal{Z})=0, (\mathcal{Z},\bar{\mathcal{Z}}) > 0 \}.
$$
Its projectivization $\cD(M)$ is identified with the Hermitian symmetric domain of type IV, namely $\Orth^+_{2,n}/ (\SO_2\times \Orth_n)$. We denote by $\Orth^+(M)$ the subgroup of $\Orth(M)$ preserving the connected component $\cD(M)$. Let $\Gamma$ be a finite index subgroup of $\Orth^+(M)$. 
\begin{definition}
Let $k$ be a non-negative integer. A modular form of weight $k$ and character $\chi: \Gamma\to \CC^*$ for $\Gamma$ is a holomorphic function $F: \mathcal{A}(M)\to \CC$ satisfying
\begin{align*}
F(t\mathcal{Z})&=t^{-k}F(\mathcal{Z}), \quad \forall t \in \CC^*,\\
F(g\mathcal{Z})&=\chi(g)F(\mathcal{Z}), \quad \forall g\in \Gamma.
\end{align*}
\end{definition}
By \cite{BB66}, the graded algebra 
$$
M_*(\Gamma)=\bigoplus_{k\in \NN} M_k(\Gamma)
$$
of modular forms for $\Gamma$ with trivial character is finitely generated over $\CC$. Moreover, if $M_*(\Gamma)$ is freely generated, then the Satake-Baily-Borel compactification of the modular variety $\cD(M)/\Gamma$ is a weighted projective space determined by $\Proj(M_*(\Gamma))$. It is also known that if $M_*(\Gamma)$ is free then $\Gamma$ is generated by reflections (see \cite{VP89}). The first example of such free algebras was determined by Igusa \cite{Igu62}, which is related to the orthogonal group of signature $(2,3)$.  By means of the theory of Weyl invariant Jacobi forms \cite{Wir92}, we proved  the freeness of 25 graded algebras of orthogonal modular forms related to root lattices in a universal method joint with B. Williams \cite{WW20}. In \cite{Wan20} we established a necessary and sufficient condition for $M_*(\Gamma)$ to be free, which is based on the existence of a remarkable modular form which vanishes exactly on the mirrors of reflections in $\Gamma$ and equals the Jacobian determinant of the $n+1$ generators.  As a continuation of our previous work \cite{WW20}, in this paper we apply the sufficient condition to some rescaled root lattices and construct 16 new free algebras of orthogonal modular forms. 

\begin{theorem}\label{MTH}
Let $M=2U\oplus L(-1)$, where $U$ is a hyperbolic plane and $L=A_1(2)$, $A_1(3)$, $A_1(4)$, $2A_1(2)$, $A_2(2)$, $A_2(3)$, $A_3(2)$, or $D_4(2)$. For each $M$, there is at least one finite index subgroup $\Gamma$ of $\Orth^+(M)$ generated by reflections such that $M_*(\Gamma)$ is freely generated over $\CC$. The generators can be constructed as additive lifts of certain holomorphic  Jacobi forms.
\end{theorem}

The proof is based on some explicit constructions of modular forms.  We first construct some modular forms whose divisor is a sum of some mirrors of reflections. The divisors of these modular forms will determine the arithmetic group $\Gamma$. We then construct basic modular forms as additive lifts with characters (see \cite{CG13}) of some algebraically independent Jacobi forms. This guarantees the non-vanishing of their Jacobian. It is enough to conclude the theorem from these constructions and the sufficient condition mentioned above. As an application of the main theorem, we prove the modularity of formal Fourier-Jacobi expansions for these reflection groups.  Such property has applications in arithmetic geometry (see \cite{BR15}).

The paper is organized as follows. In \S \ref{Sec:JFs} we introduce some necessary materials, including Jacobi forms of lattice index, additive lifts, Borcherds products, and the sufficient condition. In \S \ref{Sec:freealgebras} we construct special modular forms and prove the main theorem case by case.

\section{Preliminaries}\label{Sec:JFs}

\subsection{Jacobi forms of lattice index}
Let $M$ be an even lattice of signature $(2,n)$. We denote its dual lattice by $M^\vee$ and its discriminant group by $D(M)=M^\vee/M$. The discriminant kernel $\widetilde{\Orth}^+(M)$ is defined as the kernel of the natural homomorphism $\Orth^+ (M) \to \Orth(D(M))$.  For $v\in M$, we denote the positive generator of the ideal $(v,M)=\{(v,x): x\in M\}$ by $\div(v)$. 
For any integer $a$, we denote by $M(a)$ the lattice obtained by rescaling the bilinear form of $M$ with $a$.  

We assume that $M=2U\oplus L(-1)$, where $L$ is an even positive definite lattice.
Jacobi forms of lattice index $L$ appear naturally in the Fourier-Jacobi expansions of orthogonal modular forms for $\widetilde{\Orth}^+(M)$  at the standard 1-dimensional cusp determined by $2U$ on the tube domain 
$$
\cH(L)=\{Z=(\tau,\mathfrak{z},\omega)\in \HH\times (L\otimes\CC)\times \HH: 
(\im Z,\im Z)>0\}, 
$$
where $(\im Z,\im Z)=2\im \tau \im \omega - (\im \mathfrak{z},\im \mathfrak{z})_L$.

Following \cite{CG13}, we define Jacobi forms of rational index with character associated to $L$. As explained in \cite[Proposition 2.3]{CG13}, we only need to distinguish  Jacobi forms of integral  and half-integral index associated to an even lattice. Let $\upsilon_{\eta}$ be the multiplier system of the Dedekind $\eta$-function $\eta(\tau)=q^{1/24}\prod_{j=1}^\infty(1-q^j)$. The integral \textit{Heisenberg group} of $L$ is defined as
$$
H(L)=\left\lbrace [x,y:r]: x,y\in L, r+(x,y)/2\in \ZZ \right\rbrace.
$$
The minimal integral \textit{Heisenberg group} of $L$ is the subgroup
$$
H_s(L)=\latt{[x,0:0], [0,y:0] \vert \; x,y \in L}.
$$
We remark that  Jacobi forms can be regarded as (orthogonal) modular forms for the integral Jacobi group $\Gamma^J(L)$ which is isomorphic to the semi-direct product of $\SL_2(\ZZ)$ with $H(L)$. 
We define 
\begin{equation}
\nu([x,y:r])=e^{\pi i((x,x)/2+(y,y)/2-(x,y)/2+r )}.
\end{equation}
By \cite[\S 2]{CG13}, any finite-order character (or multiplier system) of $\Gamma^J(L)$ is determined by $\upsilon_{\eta}$ and $\nu$. We define Jacobi forms as follows.

\begin{definition}
For $k\in\frac{1}{2}\ZZ$, $t\in \frac{1}{2}\NN$ and $\chi = \upsilon_{\eta}^D\times \nu^{2t}$, where $D$ is a positive integer. A holomorphic function $\varphi : \HH \times (L \otimes \CC) \rightarrow \CC$ is 
called a {\it weakly holomorphic} Jacobi form of weight $k$ and index $t$ with character (or multiplier system) $\chi$ associated to $L$,
if it satisfies the transformation laws
\begin{align*}
\varphi \left( \frac{a\tau +b}{c\tau + d},\frac{\mathfrak{z}}{c\tau + d} 
\right)& = \chi(A) (c\tau + d)^k 
\exp{\left(i \pi t \frac{c(\mathfrak{z},\mathfrak{z})}{c 
\tau + d}\right)} \varphi ( \tau, \mathfrak{z} ), \quad \left(\begin{array}{cc} 
a & b \\ 
c & d
\end{array}
\right)   \in \SL_2(\ZZ), \\
\varphi (\tau, \mathfrak{z}+ x \tau + y)&= 
\chi([x,y:(x,y)/2])\exp{\bigl(-i \pi t ( (x,x)\tau +2(x,\mathfrak{z}))\bigr)} 
\varphi (\tau, \mathfrak{z} ), \quad x,y \in L,
\end{align*}
and if its Fourier expansion  takes the form
\begin{equation*}
\varphi ( \tau, \mathfrak{z} )=\sum_{ \substack{n\geq n_1, n\equiv \frac{D}{24}\m  \ZZ \\ \ell \in \frac{1}{2}L^{\vee} } }f(n,\ell)q^n\zeta^\ell,
\end{equation*}
where $n_1\in \ZZ$ is a constant, $q=e^{2\pi i \tau}$ and $\zeta^\ell=e^{2\pi i (\ell,
\mathfrak{z})}$. 
If $f(n,\ell) = 0$ whenever $n < 0$,
then $\varphi$ is called a {\it weak} Jacobi form.  If $ f(n,\ell) = 0 $ whenever
$ 2nt - (\ell,\ell) < 0 $ (resp. $\leq 0$),
then $\varphi$ is called a {\it holomorphic}
(resp. {\it cusp}) Jacobi form. 
\end{definition}

We denote by $J^{!}_{k,L,t}(\chi)$ (resp. $J^{w}_{k,L,t}(\chi)$, $J_{k,L,t}(\chi)$, 
$J_{k,L,t}^{\text{cusp}}(\chi)$) the vector space of weakly holomorphic Jacobi 
forms (resp. weak, holomorphic, cusp Jacobi forms) 
of weight $k$ and index $t$ with character $\chi$ for $L$. 
The classical Jacobi forms defined by Eichler--Zagier \cite{EZ85} $J_{k,t}$
are identical to the Jacobi forms $J_{k,A_1, t}$ for the lattice  $A_1=\latt{\ZZ, 2x^2}$.
The most important Jacobi form is defined by the Jacobi triple product 
$$
\vartheta(\tau,z)=q^\frac{1}{8} (\zeta^\frac{1}{2}-\zeta^{-\frac{1}{2}} )\prod_{n\geq 1} (1-q^n\zeta)(1-q^n \zeta^{-1})(1-q^n),
$$
where $q=e^{2 \pi i \tau}$ and $\zeta=e^{2 \pi i z}$, $z \in \CC$. Under the above definition, $\vartheta\in J_{\frac{1}{2},A_1,\frac{1}{2}}(\upsilon_\eta^3\times \nu)$ (see \cite{GN98}).

\subsection{Additive lifts with characters} In this subsection we recall the additive lifts with characters introduced in \cite{CG13}. These lifts are our main tool to construct generators of modular forms. We start with the raising index Hecke operators.

\begin{proposition}[see Proposition 3.1 in \cite{CG13}]\label{prop:CG1}
Let $\varphi \in J_{k,L,t}^w(\upsilon_{\eta}^D \times \nu^{2t})$ be a weak Jacobi form of weight $k$ and index $t$ for $L$. We assume that $k$ is integral and $D$ is an even divisor of $24$. If $Q=\frac{24}{D}$ is odd, we further assume that $t$ is integral. Then for any natural $m$ coprime to $Q$, we have 
$$
\varphi \lvert_{k,t}T_{-}^{(Q)}(m)(\tau,\mathfrak{z})=m^{-1}\sum_{\substack{ad=m,a>0\\ 0\leq b <d}}a^k \upsilon_{\eta}^D(\sigma_a)\varphi\left(\frac{a\tau+bQ}{d},a\mathfrak{z}\right)\in J_{k,L,mt}^w(\upsilon_{\eta}^{D x} \times \nu^{2t}),
$$
where $x,y\in \ZZ$ such that $mx+Qy=1$, 
and
$$\sigma_a= \left(\begin{array}{cc}
dx+Qdxy & -Qy \\ 
Qy & a
\end{array}  \right) \in \SL_2(\ZZ).$$ 
Moreover, if $f(n,\ell)$ are the Fourier coefficients of $\varphi$, 
then the Fourier coefficients of $\varphi \lvert_{k,t}T_{-}^{(Q)}(m)$ have the following form
$$
f_m(n,\ell)=\sum_{a \mid (n,\ell,m)}a^{k-1} v_{\eta}^D(\sigma_a)f\left( \frac{nm}{a^2},\frac{\ell}{a}\right),
$$
where $a \mid (n,\ell,m)$ means that $a\mid nQ$, $a^{-1}\ell\in \frac{1}{2}L^\vee$ and $a\mid m$.
\end{proposition}

The additive lifts is defined in terms of the above Hecke operators.

\begin{theorem}[see Theorem 3.2 in \cite{CG13}]\label{th:CG2}
Let $\varphi \in J_{k,L,t}(v_{\eta}^D \times \nu^{2t})$, $k$ be integral, $t$ be half-integral and $D$ be an even divisor of $24$. If the conductor $Q=24/D$ is odd, we further assume that $t$ is integral.  Let $G_k$ be the $\SL_2(\ZZ)$ Eisenstein series of weight $k$, normalized such that the Fourier coefficient at $q$ is 1. Then the function 
$$ 
\Grit(\varphi)(Z)=f(0,0)G_k(\tau)+\sum_{\substack{m\equiv 1\m Q\\ m>0}}\varphi \lvert_{k,t}T_{-}^{(Q)}(m)(\tau,\mathfrak{z}) e^{2\pi i \frac{m}{Q}\omega}
$$
is a modular form of weight $k$ with respect to the discriminant kernel $\widetilde{\Orth}^+(2U\oplus L(-Qt))$ with a character $\chi$ of order $Q$ induced by $v_\eta^D$, the restriction of $\nu^{2t}$ on $H_s(L)$, and the relations $\chi ([0,0:r])=e^{2\pi i \frac{r}{Q}}$, $\chi(V)=1$, where $V: (\tau,\mathfrak{z}, \omega) \mapsto (\omega,\mathfrak{z},\tau)$.
\end{theorem}

\subsection{Borcherds products}
For any negative norm vector $r\in M^\vee$, the hyperplane
\begin{equation}
 \cD_r(M)=r^\perp\cap \cD(M)=\{ [\mathcal{Z}]\in \cD(M) : (\mathcal{Z},r)=0\}
\end{equation}
is called the rational quadratic divisor associated to $r$.  (We will write $\cD_r$ instead of $\cD_r(M)$  if there is no confusion.)
The reflection fixing $\cD_r(M)$ is defined as
\begin{equation}
\sigma_r(x)=x-\frac{2(r,x)}{(r,r)}r,  \quad x\in M.
\end{equation}
The hyperplane $\cD_r(M)$ is called the mirror of $\sigma_r$.
A primitive vector $r\in M$ of negative norm is called reflective if $\sigma_r\in\Orth^+(M)$, in which case we call $\cD_r(M)$ a reflective divisor. A primitive vector $l\in M$ with $(l,l)=-2d$ is reflective if and only if $\div(l)=2d$ or $d$. A modular form $F$ for $\Gamma<\Orth^+(M)$ is called reflective if its zero divisor is a sum of some reflective divisors. Reflective modular forms are very exceptional and have applications in generalized Kac--Moody algebras, algebraic geometry, and reflection groups (see e.g. a survey \cite{Gri18}). The celebrated Borcherds products \cite{Bor98} provide a powerful method for the  construction of reflective modular forms.  In this paper, we use the variant of Borcherds products attributed to Gritsenko-Nikulin in the context of Jacobi forms. 

\begin{theorem}[see Theorem 4.2 in \cite{Gri18})]
We fix an ordering $\ell >0$ in $L^\vee$ in a way similar to positive root systems (see the bottom of page $825$ in \cite{Gri18}).
Let 
$$
\varphi(\tau,\mathfrak{z})=\sum_{n\in\ZZ}\sum_{\ell\in L^\vee}f(n,\ell)q^n \zeta^\ell \in J^{!}_{0,L,1}.
$$
Assume that $f(n,\ell)\in \ZZ$ for all $2n-(\ell,\ell)\leq 0$.   We set
\begin{align*}
&A=\frac{1}{24}\sum_{\ell\in L^\vee}f(0,\ell),& &\vec{B}=\frac{1}{2}\sum_{\ell>0} f(0,\ell)\ell,& &C=\frac{1}{2\rank(L)}\sum_{\ell\in L^\vee}f(0,\ell)(\ell,\ell).&
\end{align*}
Then the infinite product
$$
\Borch(\varphi)(Z)=q^A \zeta^{\vec{B}} \xi^C\prod_{\substack{n,m\in\ZZ, \ell\in L^\vee\\ (n,\ell,m)>0}}(1-q^n \zeta^\ell \xi^m)^{f(nm,\ell)}
$$
defines a meromorphic modular form of weight $f(0,0)/2$ for $\widetilde{\Orth}^+(2U\oplus L(-1))$ with a character $\chi$,
where $Z= (\tau,\mathfrak{z}, \omega) \in \cH(L)$, $q=\exp(2\pi i \tau)$, $\zeta^\ell=\exp(2\pi i (\ell, \mathfrak{z}))$, $\xi=\exp(2\pi i \omega)$, the notation $(n,\ell,m)>0$ means that either $m>0$, or $m=0$ 
and $n>0$, or $m=n=0$ and $\ell<0$. The character $\chi$ is induced by
\begin{align*}
&\chi \lvert_{\SL_2(\ZZ)}=v_\eta^{24A},& &\chi \lvert_{H(L)}([\lambda,\mu; r])=e^{\pi i C((\lambda,\lambda)+(\mu, \mu )- (\lambda, \mu ) +2r)},& & \chi(V)=(-1)^D,&
\end{align*}
where $V: (\tau,\mathfrak{z}, \omega) \to (\omega,\mathfrak{z},\tau)$ and $D=\sum_{n<0}\sigma_0(-n) f(n,0)$. The poles and zeros of $\Borch(\varphi)$ lie on the rational quadratic divisors $\cD_v(2U\oplus L(-1))$, where $v\in 2U\oplus L^\vee(-1)$ is a primitive vector of negative norm. The multiplicity of this divisor is given by 
$$ 
\operatorname{mult} \cD_v(2U\oplus L(-1)) = \sum_{d\in \ZZ,d>0 } f(d^2n,d\ell),
$$
where $n\in\ZZ$, $\ell\in L^\vee$ such that $(v,v)=2n-(\ell,\ell)$ and $v\equiv \ell\mod 2U\oplus L(-1)$.
Moreover, we have
$$ 
\Borch(\varphi)=\psi_{L,C}(\tau,\mathfrak{z})\xi^C \exp \left(-\Grit(\varphi) \right),
$$
where the leading Fourier-Jacobi coefficient is a generalized theta block of the form
\begin{equation*}
\psi_{L,C}(\tau,\mathfrak{z})=\eta(\tau)^{f(0,0)}\prod_{\ell >0}\left(\frac{\vartheta(\tau,(\ell,\mathfrak{z}))}{\eta(\tau)} \right)^{f(0,\ell)}.
\end{equation*}
\end{theorem}

\subsection{The sufficient condition to be free algebras}
We state the sufficient condition for the graded algebra of modular forms being free proved in \cite{Wan20}, which plays a vital role in the proof of Theorem \ref{MTH}. 

\begin{theorem}[see Theorem 5.1 in \cite{Wan20}]\label{th:converseJacobian}
Let $M$ be an even lattice of signature $(2,n)$ with $n\geq 3$ and $\Gamma$ be a finite index subgroup of $\Orth^+(M)$. If there exists a modular form $F$ (with a character) on $\Gamma$  which vanishes exactly on all mirrors of reflections in $\Gamma$ with multiplicity one and equals the Jacobian determinant of $n+1$ certain modular forms on $\Gamma$ (with trivial character), then the graded algebra $M_*(\Gamma)$ is freely generated by the $n+1$ modular forms. Moreover, the group $\Gamma$ is generated by all reflections whose mirrors are contained in the divisor of the modular form $F$.
\end{theorem}

In the above theorem, the Jacobian determinant of modular forms means the Rankin-Cohen-Ibukiyama differential operators introduced in \cite{AI05}. A precise definition and some basic properties can be found in \cite[Theorem 2.5]{Wan20}. We mention that the Jacobian is not identically zero if and only if these forms are algebraically independent over $\CC$. In addition, the Jacobian has the determinant character on $\Gamma$ and thus vanishes on all mirrors of reflections in $\Gamma$.  Hence it is rather easy to verify if a reflective Borcherds product equals a Jacobian of some modular forms. For example, if $F$ is a Borcherds product vanishing exactly on all mirrors of reflections in $\Gamma$ with multiplicity one and the same weight modular form $J$ is a Jacobian of $n+1$ modular forms with trivial character on $\Gamma$, then they are equal up to a constant multiple because the quotient $J/F$ defines a holomorphic modular form of weight $0$.

\section{Free algebras of modular forms}\label{Sec:freealgebras}
Let  $L=A_1(2)$, $A_1(3)$, $A_1(4)$, $2A_1(2)$, $A_2(2)$, $A_2(3)$, $A_3(2)$, or $D_4(2)$. In this section,  we find one reflection subgroup of $\Orth^+(2U\oplus A_1(-4))$ and three reflection subgroups of $\Orth^+(2U\oplus D_4(-2))$ such that the associated algebras of modular forms are freely generated.
For every remaining lattice $L$, we determine two reflection subgroups $\Gamma$ of $\Orth^+(2U\oplus L(-1))$ such that $M_*(\Gamma)$ is freely generated.  This gives a proof of Theorem \ref{MTH}.

We first introduce some basic results which will be used later. It is easy to check that the above eight lattices $L$ satisfy the following $\norm_2$ condition: (see \cite[\S 4]{GN18})
\begin{equation}
\norm_2:\ \forall\, \bar{c} \in L^\vee/L \quad \exists\, h_c \in \bar{c} 
\quad \text{such that} \quad (h_c,h_c)\leq 2.
\end{equation} 
Let $M=2U\oplus L(-1)$. For every above $L$, we have the following simple facts:
\begin{enumerate}
\item The natural homomorphism $\Orth(L)\to \Orth(L^\vee/L)$ is surjective. Thus if the input is invariant up to a character under the action of $\Orth(L)$ then the additive lift is a modular form for $\Orth^+(M)$ with a character. If the input is invariant under $\Orth(L)$ then the Borcherds product is a modular form for $\Orth^+(M)$ with a character. 
\item (\cite[Lemma 3.5]{GW20}) Let $\phi\in J_{0,L,1}^w$ with integral Fourier coefficients.  The divisors of $\Borch(\phi)$ are determined entirely by the $q^0$-term of $\phi$, i.e., any divisor of $\Borch(\phi)$ is of the form $\cD_{(0,0,\ell,1,0)}$ up to the $\widetilde{\Orth}^+(M)$-action.
\item (\cite[Theorem 4.4]{GN18}) The pull-back of the Borcherds form $\Phi_{12}$ for $\II_{2,26}$ gives a modular form $\Phi_{12,L}$ of weight 12 and a character of order 2 for $\Orth^+(M)$ with complete 2-divisor, i.e. $\Phi_{12,L}$ vanishes exactly with multiplicity one on $\cD_r$ for all $r\in M$ with $(r,r)=-2$. 

\cite[Theorem 4.4]{GN18} only asserts that $\Phi_{12,L}$ is a modular form for $\widetilde{\Orth}^+(M)$ with the character $\det$. But $\Phi_{12,L}$ is a Borcherds product of a vector-valued nearly holomorphic modular form $f$ for the Weil representation of $\SL_2(\ZZ)$ attached to $D(L)$. We replace $f$ with the $\Orth(D(L))$-invariant modular form $\frac{1}{\lvert\Orth(D(L))\rvert}\sum_{\gamma\in \Orth(D(L))} f\vert \gamma$ as the input of Borcherds product. Then we get a modular form of weight 12 with the same divisor for $\Orth^+(M)$. Thus the two modular forms are the same and we prove the above claim.
\item There are Jacobi Eisenstein series of weights 4 and 6 for $L$, i.e. $E_{4,L}=1+O(q)\in J_{4,L,1}$ and $E_{6,L}=1+O(q)\in J_{6,L,1}$. We can assume that they are invariant with respect to $\Orth(L)$. Then their additive lifts denoted by $\cE_{4,L}$ and $\cE_{6,L}$ are modular forms of weights 4 and 6 for $\Orth^+(M)$, respectively. 
\end{enumerate}

In the next subsections we use Theorem \ref{th:converseJacobian} to prove Theorem \ref{MTH}. For every lattice $L$ of rank $l$, we first construct $l+3$ basic modular forms as additive lifts introduced in Theorem \ref{th:CG2}. The first non-zero Fourier--Jacobi coefficients of the $l+3$ additive lifts are respectively the two $\SL_2(\ZZ)$ Eisenstein series $E_4$, $E_6$ and the $l+1$ generators of the bigraded ring of Weyl invariant weak Jacobi forms introduced in \cite{Wir92}. The algebraic independence of the $l+1$ elementary  Jacobi forms and the two Eisenstein series yield the non-vanishing of the Jacobian of the $l+3$ additive lifts. We then construct some basic reflective modular forms using Borcherds products.  By considering the product of some basic reflective modular forms, we further construct a reflective modular form $F$ which has the same weight as the Jacobian. The reflections corresponding to the divisor of $F$ generate a group $\Gamma$. We show that these additive lifts are modular forms with trivial character for $\Gamma$, which implies that the Jacobian is equal to $F$ up to a constant multiple. This proves the freeness of $M_*(\Gamma)$ by Theorem \ref{th:converseJacobian}. We only give a full proof of the case $L=A_1(2)$ because the proofs of other cases are similar.

\subsection{The \texorpdfstring{$A_1$}{A1} case}
We consider the cases of $2U\oplus A_1(-n)$ for $n=2,3,4$. We know from \cite{EZ85} that the bigraded ring of weak Jacobi forms of even weight and integral index for $A_1$ is freely generated over the ring $M_*(\SL_2(\ZZ))$ of $\SL_2(\ZZ)$-modular forms by two forms of index 1 and weights $-2$ and 0. We fix the following model of $A_1$
$$
A_1=\latt{\ZZ \cdot e, e^2=2}, \quad \mathfrak{z}=ze, z\in \CC, \; \zeta=e^{2\pi i z}.
$$
Under this coordinate, the mentioned generators have the following Fourier expansions:
\begin{align*}
\phi_{0,1}&=\zeta + \zeta^{-1} +10 \in J_{0,1}^w,\\
\phi_{-2,1}&=\phi_{-1,\frac{1}{2}}^2=\zeta + \zeta^{-1} -2 \in J_{-2,1}^w, \quad  \phi_{-1,\frac{1}{2}}=\frac{\vartheta(\tau,z)}{\eta^3(\tau)}. 
\end{align*}
We also need three weak Jacobi forms of weight 0 which are the generators of the ring of weak Jacobi forms of weight 0 and integral index with integral Fourier coefficients (see \cite[\S 3.1]{GW18}).
\begin{align*}
\phi_{0,2}&=\zeta + \zeta^{-1} + 4 \in J_{0,2}^w,\\
\phi_{0,3}&=\zeta + \zeta^{-1} + 2 \in J_{0,3}^w,\\
\phi_{0,4}&=\zeta + \zeta^{-1} + 1 \in J_{0,4}^w.
\end{align*}
We next construct some modular forms using these basic Jacobi forms. Most of them have been constructed in \cite{GN98}. But for convenience we repeat their constructions.

\subsubsection{The case of $n=2$}
There are three Borcherds products:
\begin{enumerate}
\item $\Borch(\phi_{0,2})\in M_2(\Orth^+(2U\oplus A_1(-2)), \chi_1)$. Its divisor is the sum of $\cD_v$ with multiplicity one for $v\in 2U\oplus A_1(-2)$ primitive, $(v,v)=-4$, and $\div(v)=4$. 
\item $\Borch(\phi^{(1)}_{0,2})\in M_{11}(\Orth^+(2U\oplus A_1(-2)),\chi_2)$. Its divisor is the sum of $\cD_u$ and $\cD_v$ with multiplicity one, where $u\in 2U\oplus A_1(-2)$ primitive, $(u,u)=-4$ and $\div(u)=2$,  and $v\in 2U\oplus A_1(-2)$ primitive, $(v,v)=-4$ and $\div(v)=4$. The function $\phi^{(1)}_{0,2}$ is defined as
$$
\phi^{(1)}_{0,2}=\phi_{0,1}^2-20\phi_{0,2}=\zeta^{\pm 2} +22 + O(q).
$$
\item $\Phi_{12, A_1(2)} \in M_{12}(\Orth^+(2U\oplus A_1(-2)), \chi_3)$. Its divisor is the sum of $\cD_r$ with multiplicity one for $r\in 2U\oplus A_1(-2)$ with $(r,r)=-2$. 
\end{enumerate}
We also construct two additive lifts
\begin{align*}
\Grit(\eta^{12}\phi_{0,1})\in M_6(\Orth^+(2U\oplus A_1(-2)), \chi_4),\\
\Grit(\eta^{6}\phi_{-1,\frac{1}{2}})\in M_2(\Orth^+(2U\oplus A_1(-2)), \chi_5).
\end{align*}

These five characters can be worked out. But characters are not important for us in this paper. Thus we do not write them explicitly. For the remaining cases, when we say that a modular form has a character $\chi$, it only means that this form has a non-trivial character. The same symbol $\chi$ may stand for different characters of different groups.

It is easy to check that $\Grit(\eta^{6}\phi_{-1,\frac{1}{2}})$ vanishes on all divisors of $\Borch(\phi_{0,2})$ (see \cite{GN98} or \cite{GW20}). We then conclude that $\Grit(\eta^{6}\phi_{-1,\frac{1}{2}})=\Borch(\phi_{0,2})$.

\begin{theorem}\label{th:A1-2}
Let $\Gamma_{2,4}(A_1(2))$ be the subgroup of $\Orth^+(2U\oplus A_1(-2))$ generated by reflections associated to vectors of types $r$, $u$ and $v$ above.  Let $\Gamma_{2,4'}(A_1(2))$ be the subgroup of $\Orth^+(2U\oplus A_1(-2))$ generated by reflections associated to vectors of types $r$ and $u$ above.  
\begin{enumerate}
\item The graded algebra $M_*(\Gamma_{2,4}(A_1(2)))$ is freely generated by $\cE_{4, A_1(2)}$, $\cE_{6, A_1(2)}$, $\Grit(\eta^{12}\phi_{0,1})$ and $\Grit(\eta^{6}\phi_{-1,\frac{1}{2}})^2$. The Jacobian determinant equals $\Borch(\phi^{(1)}_{0,2})\Phi_{12, A_1(2)}$ up to a constant. $(23-3=4+6+6+4)$
\item The graded algebra $M_*(\Gamma_{2,4'}(A_1(2)))$ is freely generated by $\cE_{4, A_1(2)}$, $\cE_{6, A_1(2)}$, $\Grit(\eta^{12}\phi_{0,1})$ and $\Grit(\eta^{6}\phi_{-1,\frac{1}{2}})$. The Jacobian determinant equals $\Borch(\phi^{(1)}_{0,2})\Phi_{12, A_1(2)}/\Borch(\phi_{0,2})$ up to a constant. $(21-3=4+6+6+2)$
\end{enumerate}
\end{theorem}

\begin{proof}
We first prove the part (2).
Let $\Gamma_0$ be the intersection of the kernels of the two characters $\chi_4$ and $\chi_5$. Then it is a finite index subgroup of $\Orth^+(2U\oplus A_1(-2))$. Moreover, the four forms $\cE_{4, A_1(2)}$, $\cE_{6, A_1(2)}$, $\Grit(\eta^{12}\phi_{0,1})$ and $\Grit(\eta^{6}\phi_{-1,\frac{1}{2}})$ are modular forms for $\Gamma_0$ with trivial character, and their leading Fourier-Jacobi  coefficients are respectively the $\SL_2(\ZZ)$-Eisenstein series $E_4$, $E_6$, and the Jacobi forms $\eta^{12}\phi_{0,1}$, $\eta^{6}\phi_{-1,\frac{1}{2}}$. Since $E_4$, $E_6$, $\phi_{0,1}$ and $\phi_{-1,\frac{1}{2}}$ are algebraically independent over $\CC$, it follows that the four additive lifts are also algebraically independent over $\CC$. Therefore their Jacobian is not zero and defines a modular form of weight 21 on $\Gamma_0$ with the determinant character. 

We now show that all reflections associated to vectors of types $r$ and $u$ are contained in $\Gamma_0$. If there is a vector $r$ such that $\sigma_r\not\in \Gamma_0$, then $\chi_4(\sigma_r)\neq 1$ or $\chi_5(\sigma_r)\neq 1$. It follows that $\Grit(\eta^{12}\phi_{0,1})$ or $\Grit(\eta^{6}\phi_{-1,\frac{1}{2}})$ vanishes on $\cD_r$. Thus the quotient of $\Grit(\eta^{12}\phi_{0,1})$ or $\Grit(\eta^{6}\phi_{-1,\frac{1}{2}})$ by $\Phi_{12,A_1(2)}$ is holomorphic, which leads to a contradiction. Therefore all reflections associated to vectors of type $r$ are contained in $\Gamma_0$.  Similarly, we prove that all reflections associated to vectors of type $u$ are contained in $\Gamma_0$. 

From the above, we see that $\Borch(\phi^{(1)}_{0,2})\Phi_{12, A_1(2)}$ is a modular form (with a character) of weight 21 on $\Gamma_0$ whose divisor is a sum of some mirrors of reflections in $\Gamma_0$ with multiplicity one.  Recall that the Jacobian of generators vanishes exactly on all mirrors of reflections in $\Gamma_0$.  By comparing the weight and the divisor, we assert that this modular form equals the Jacobian of the four additive lifts up to a constant multiple. We then conclude from Theorem \ref{th:converseJacobian} that the graded algebra $M_*(\Gamma_0)$ is freely generated by the four additive lifts. Moreover, $\Gamma_0$ is generated by all reflections whose mirrors are contained in the divisor of the Jacobian. This implies that $\Gamma_0$ coincides with $\Gamma_{2,4'}(A_1(2))$.

Similarly, we show that $\Gamma_{2,4}(A_1(2))$ is the intersection of kernels of $\chi_4$ and $\chi_5^2$. We then prove that $M_*(\Gamma_{2,4}(A_1(2)))$ is freely generated by given generators. 
\end{proof}

As an application, we prove the modularity of formal Fourier-Jacobi expansions for $\Gamma_{2,4}(A_1(2))$.  
The space of formal Fourier-Jacobi expansions of integral weight $2k$ for $\Gamma_{2,4}$ is defined as
$$
FM_{2k}(\Gamma_{2,4}(A_1(2)))=\left\{ \sum_{m=0}^{\infty} \psi_m \xi^{m/2} \in \prod_{m=0}^\infty J_{2k,m}(v_\eta^{12m}) : f_m(n,r)=f_{2n}(m/2,r), \forall n,r,m \right\}
$$
where $f_m(n,r)$ are Fourier coefficients of $\psi_m$. 
\begin{corollary}\label{Cor:FFJ}
The following map is an isomorphism
\begin{align*}
M_{2k}(\Gamma_{2,4}(A_1(2))) &\to FM_{2k}(\Gamma_{2,4}(A_1(2))),\\
F &\mapsto \text{Fourier-Jacobi expansion of $F$}.
\end{align*}
\end{corollary}

\begin{proof}
Firstly, the map is defined well by the shape of generators in the above theorem. The injectivity of the map is obvious. We next show that the map is surjective. On the one hand, by the symmetry of formal Fourier-Jacobi expansions and an argument similar to \cite[\S 3]{WW20}, we have
$$
\dim FM_{2k}(\Gamma_{2,4}(A_1(2))) \leq \sum_{m=0}^\infty \dim J_{2k,m}(v_\eta^{12m})[m/2]= \sum_{m=0}^\infty \dim J_{2k-6m,m}^w,
$$
where 
$$
J_{2k,m}(v_\eta^{12m})[m/2]=\{ \phi \in J_{2k,m}(v_\eta^{12m}) : \phi= O(q^{m/2}) \}.
$$
We remark that the above infinite sums are in fact finite sums because the spaces  $J_{2k,m}(v_\eta^{12m})[m/2]$ and $J_{2k-6m,m}^w$ are both trivial when $m$ is sufficiently large. 
On the other hand, for an arbitrary $\phi_m\in J_{2k,m}(v_\eta^{12m})[m/2]$, the function $\phi_m/ \eta^{12m}$ is a weak Jacobi form of weight $2k-6m$ and index $m$ with trivial character. Therefore there exists a polynomial $P$ in four variables over $\CC$ such that $\phi_m/ \eta^{12m}=P(E_4,E_6,\phi_{0,1},\phi_{-2,1})$. Then  $P(\cE_{4,A_1(2)},\cE_{6,A_1(2)}, \Grit(\eta^{12}\phi_{0,1}), \Grit(\eta^{6}\phi_{-1,\frac{1}{2}})^2)$ gives an orthogonal modular form of weight $2k$ whose first non-zero Fourier-Jacobi coefficient is $\phi_m\cdot\xi^{m/2}$. It follows that 
$$
\dim M_{2k}(\Gamma_{2,4}(A_1(2))) \geq \sum_{m=0}^\infty \dim J_{2k,m}(v_\eta^{12m})[m/2]= \sum_{m=0}^\infty \dim J_{2k-6m,m}^w.
$$
We then derive 
$$
\dim M_{2k}(\Gamma_{2,4}(A_1(2))) =\dim FM_{2k}(\Gamma_{2,4}(A_1(2))) = \sum_{m=0}^\infty \dim J_{2k-6m,m}^w,
$$
which completes the proof.
\end{proof}

\subsubsection{The case of $n=3$}
There are three Borcherds products:
\begin{enumerate}
\item $\Borch(\phi_{0,3})\in M_1(\Orth^+(2U\oplus A_1(-3)), \chi)$. Its divisor is the sum of $\cD_v$ with multiplicity one for $v\in 2U\oplus A_1(-3)$ primitive, $(v,v)=-6$, and $\div(v)=6$. 
\item $\Borch(\phi^{(1)}_{0,3})\in M_{7}(\Orth^+(2U\oplus A_1(-3)),\chi)$. Its divisor is the sum of $\cD_u$ and $\cD_v$ with multiplicity one, where $u\in 2U\oplus A_1(-3)$ primitive, $(u,u)=-6$ and $\div(u)=3$,  and $v\in 2U\oplus A_1(-3)$ primitive, $(v,v)=-6$ and $\div(v)=6$. The function $\phi^{(1)}_{0,3}$ is defined as
$$
\phi^{(1)}_{0,3}=\phi_{0,1}\phi_{0,2}-14\phi_{0,3}=\zeta^{\pm 2} +14 + O(q).
$$
\item $\Phi_{12, A_1(3)} \in M_{12}(\Orth^+(2U\oplus A_1(-3)), \chi)$. Its divisor is the sum of $\cD_r$ with multiplicity one for $r\in 2U\oplus A_1(-3)$ with $(r,r)=-2$. 
\end{enumerate}
We also construct two additive lifts
\begin{align*}
\Grit(\eta^{8}\phi_{0,1})\in M_4(\Orth^+(2U\oplus A_1(-3)), \chi),\\
\Grit(\eta^{4}\phi_{-1,\frac{1}{2}})\in M_1(\Orth^+(2U\oplus A_1(-3)), \chi).
\end{align*}

We remark that $\Grit(\eta^{4}\phi_{-1,\frac{1}{2}})=\Borch(\phi_{0,3})$ (see \cite{GN98}).

Similar to Theorem \ref{th:A1-2}, it is easy to prove the following result.
\begin{theorem}
Let $\Gamma_{2,6}(A_1(3))$ be the subgroup of $\Orth^+(2U\oplus A_1(-3))$ generated by reflections associated to vectors of types $r$, $u$ and $v$ above.  Let $\Gamma_{2,6'}(A_1(3))$ be the subgroup of $\Orth^+(2U\oplus A_1(-3))$ generated by reflections associated to vectors of types $r$ and $u$ above.  
\begin{enumerate}
\item The graded algebra $M_*(\Gamma_{2,6}(A_1(3)))$ is freely generated by $\cE_{4, A_1(3)}$, $\cE_{6, A_1(3)}$, $\Grit(\eta^{8}\phi_{0,1})$ and $\Grit(\eta^{4}\phi_{-1,\frac{1}{2}})^2$. The Jacobian determinant equals $\Borch(\phi^{(1)}_{0,3})\Phi_{12, A_1(3)}$ up to a constant. $(19-3=4+6+4+2)$
\item The graded algebra $M_*(\Gamma_{2,6'}(A_1(3)))$ is freely generated by $\cE_{4, A_1(3)}$, $\cE_{6, A_1(3)}$, $\Grit(\eta^{8}\phi_{0,1})$ and $\Grit(\eta^{4}\phi_{-1,\frac{1}{2}})$. The Jacobian determinant equals $\Borch(\phi^{(1)}_{0,3})\Phi_{12, A_1(3)}/\Borch(\phi_{0,3})$ up to a constant. $(18-3=4+6+4+1)$
\end{enumerate}
\end{theorem}

Similar to Corollary \ref{Cor:FFJ}, for $\Gamma_{2,6}(A_1(3))$ we prove the modularity of formal Fourier-Jacobi expansions  defined as
$$
FM_{2k}(\Gamma_{2,6}(A_1(3)))=\left\{ \sum_{m=0}^{\infty} \psi_m \xi^{m/3} \in \prod_{m=0}^\infty J_{2k,m}(v_\eta^{8m}) : f_m(n,r)=f_{3n}(m/3,r), \forall n,r,m \right\}.
$$
\begin{corollary}
The following map is an isomorphism
\begin{align*}
M_{2k}(\Gamma_{2,6}(A_1(3))) &\to FM_{2k}(\Gamma_{2,6}(A_1(3))),\\
F &\mapsto \text{Fourier-Jacobi expansion of $F$}.
\end{align*}
Moreover, we have
$$
\dim M_{2k}(\Gamma_{2,6}(A_1(3))) = \sum_{m=0}^\infty \dim J_{2k-4m,m}^w.
$$
\end{corollary}

\subsubsection{The case of $n=4$}
Similarly, there are three Borcherds products and two additive lifts:
\begin{enumerate}
\item $\Borch(\phi_{0,4})\in M_\frac{1}{2}(\Orth^+(2U\oplus A_1(-4)), \chi_{1/2})$, where $\chi_{1/2}$ is a multiplier system. Its divisor is the sum of $\cD_v$ with multiplicity one for $v\in 2U\oplus A_1(-4)$ primitive, $(v,v)=-8$ and $\div(v)=8$. 
\item $\Borch(\phi^{(1)}_{0,4})\in M_{5}(\Orth^+(2U\oplus A_1(-4)),\chi)$. Its divisor is the sum of $\cD_u$ and $\cD_v$ with multiplicity one, where $u\in 2U\oplus A_1(-4)$ primitive, $(u,u)=-8$ and $\div(u)=4$,  and $v\in 2U\oplus A_1(-4)$ primitive, $(v,v)=-8$ and $\div(v)=8$. The function $\phi^{(1)}_{0,4}$ is defined as
$$
\phi^{(1)}_{0,4}=\phi_{0,1}\phi_{0,3}-12\phi_{0,4}=\zeta^{\pm 2} +10 + O(q).
$$
\item $\Phi_{12, A_1(4)} \in M_{12}(\Orth^+(2U\oplus A_1(-4)), \chi)$. Its divisor is the sum of $\cD_r$ with multiplicity one for $r\in 2U\oplus A_1(-4)$ with $(r,r)=-2$. 
\end{enumerate}

\begin{align*}
\Grit(\eta^{6}\phi_{0,1})\in M_3(\Orth^+(2U\oplus A_1(-4)), \chi),\\
\Grit(\eta^{3}\phi_{-1,\frac{1}{2}})\in M_\frac{1}{2}(\Orth^+(2U\oplus A_1(-4)), \chi).
\end{align*}

The above Jacobi form $\eta^{3}\phi_{-1,\frac{1}{2}}=\vartheta$ has weight $\frac{1}{2}$. But we can still define its additive lift (see \cite{GN98}). It was proved in \cite{GN98} that $\Grit(\eta^{3}\phi_{-1,\frac{1}{2}})=\Borch(\phi_{0,4})$.

Similarly, we have the following theorem and corollary.

\begin{theorem}
Let $\Gamma_{2,8}(A_1(4))$ be the subgroup of $\Orth^+(2U\oplus A_1(-4))$ generated by reflections associated to vectors of types $r$, $u$ and $v$ above.  The graded algebra $M_*(\Gamma_{2,8}(A_1(4)))$ is freely generated by $\cE_{4, A_1(4)}$, $\cE_{6, A_1(4)}$, $\Grit(\eta^{6}\phi_{0,1})$ and $\Grit(\eta^{3}\phi_{-1,\frac{1}{2}})^2$. The Jacobian determinant equals $\Borch(\phi^{(1)}_{0,4})\Phi_{12, A_1(4)}$ up to a constant. $(17-3=4+6+3+1)$
\end{theorem}

The formal Fourier-Jacobi expansions  for $\Gamma_{2,8}(A_1(4))$ are defined as
$$
FM_k(\Gamma_{2,8}(A_1(4)))=\left\{ \sum_{m=0}^\infty \psi_m \xi^{m/4} \in \prod_{m=0}^\infty J_{k,m}^{\text{ev}}(v_\eta^{6m}) : f_m(n,r)=f_{4n}(m/4,r), \forall n,r,m \right\},
$$
where $J_{k,m}^{\text{ev}}(v_\eta^{6m})=\{\phi \in J_{k,m}(v_\eta^{6m}): \phi(\tau,z)=\phi(\tau,-z)\}$. We remark that $J_{k,m}^{\text{ev}}(v_\eta^{6m})=\{0\}$ if $k-m$ is odd. Let $F\in M_k(\Gamma_{2,8}(A_1(4)))$. Since $F$ is invariant under the reflection associated to vectors of type $v$, every Fourier-Jacobi coefficient is invariant under the action $z\mapsto -z$. Thus the map in the following corollary is defined well and the proof of the corollary is similar to the previous cases.
\begin{corollary}
The following map is an isomorphism
\begin{align*}
M_{k}(\Gamma_{2,8}(A_1(4))) &\to FM_{k}(\Gamma_{2,8}(A_1(4))),\\
F &\mapsto \text{Fourier-Jacobi expansion of $F$}.
\end{align*}
Moreover, we have
$$
\dim M_{k}(\Gamma_{2,8}(A_1(4))) = \sum_{\substack{m\in \NN\\m\equiv k \m 2}} \dim J_{k-3m,m}^w.
$$
\end{corollary}

\begin{remark}
In \cite{Aok16}, Aoki determined some algebras of Siegel paramodular forms of levels 2, 3, 4. The modular groups in his paper can be realized as some congruence subgroups of $\Orth^+(2U\oplus A_1(-n))$ for $n=2,3,4$, but they are bigger than our reflection groups. Thus Aoki did not get free algebras of modular forms. It is possible to recover Aoki's theorems using our results in this subsection. Besides, the modularity of formal Fourier-Jacobi expansions for Siegel paramodular groups of levels 2, 3 and 4 has been proved in \cite{IPY13}. Their results can be covered by our results because we here prove the modularity for smaller groups. 
\end{remark}

\subsection{The \texorpdfstring{$2A_1$}{2A1} case}
We consider the case of $2U\oplus 2A_1(-2)$.
The bigraded ring of weak Jacobi forms invariant under the orthogonal group $\Orth(2A_1)$ (with respect to the lattice variable $\mathfrak{z}$) is freely generated over $M_*(\SL_2(\ZZ))$ by three forms of index one $\phi_{0,2A_1}$, $\phi_{-2,2A_1}$ and $\phi_{-4,2A_1}$. In fact, this type of Jacobi forms is the so-called Weyl invariant Jacobi forms associated to the root system $B_2$ (see \cite{Wir92} and \cite[\S 2]{WW20}). We fix the model of $2A_1$:
$$
2A_1=\latt{\ZZ e_1+\ZZ e_2, e_1^2=2, e_2^2=2}, \; \mathfrak{z}=z_1e_1+z_2e_2, \; \zeta_1=e^{2\pi i z_1}, \; \zeta_2=e^{2\pi i z_2}.
$$
We construct the three generators in terms of the basic Jacobi forms for $A_1$:
\begin{align*}
\phi_{0,2A_1}&=\frac{1}{12}\left[\phi_{0,1}(\tau,z_1)\phi_{0,1}(\tau,z_2) -E_4(\tau)\phi_{-2,1}(\tau,z_1)\phi_{-2,1}(\tau,z_2) \right]\\
&=\zeta_1^{\pm 1} + \zeta_2^{\pm 1} +8+O(q) \in J_{0,2A_1,1}^{w,\Orth(2A_1)},\\
\phi_{-2,2A_1}&=\frac{1}{2}\left[\phi_{0,1}(\tau,z_1)\phi_{-2,1}(\tau,z_2) + \phi_{-2,1}(\tau,z_1)\phi_{0,1}(\tau,z_2) \right]\\
&=(\zeta_1\zeta_2)^{\pm 1}+ (\zeta_1\zeta_2^{-1})^{\pm 1}+ 4(\zeta_1^{\pm 1} + \zeta_2^{\pm 1}) -20+O(q)\in J_{-2,2A_1,1}^{w,\Orth(2A_1)},\\
\phi_{-4,2A_1}&=\phi_{-2,1}(\tau,z_1)\phi_{-2,1}(\tau,z_2) \\
&=(\zeta_1\zeta_2)^{\pm 1}+ (\zeta_1\zeta_2^{-1})^{\pm 1} -2(\zeta_1^{\pm 1} + \zeta_2^{\pm 1}) +4 +O(q) \in J_{-4,2A_1,1}^{w,\Orth(2A_1)}.
\end{align*}
As polynomial combinations of the three generators over $M_*(\SL_2(\ZZ))$, it is easy to construct the following unique weak Jacobi forms with given Fourier expansions:
\begin{align*}
\phi_{0,2A_1(2)}&=\zeta_1^{\pm 1} + \zeta_2^{\pm 1} +2+O(q) \in J_{0,2A_1,2}^{w,\Orth(2A_1)},\\
\phi^{(1)}_{0,2A_1(2)}&=\zeta_1^{\pm 2} + \zeta_2^{\pm 2} +20+O(q) \in J_{0,2A_1,2}^{w,\Orth(2A_1)},\\
\phi^{(2)}_{0,2A_1(2)}&=(\zeta_1\zeta_2)^{\pm 1}+ (\zeta_1\zeta_2^{-1})^{\pm 1}+8+O(q)\in J_{0,2A_1,2}^{w,\Orth(2A_1)}.
\end{align*}

We have four Borcherds products:
\begin{enumerate}
\item $\Borch(\phi_{0,2A_1(2)})\in M_1(\Orth^+(2U\oplus 2A_1(-2)), \chi)$. Its divisor is the sum of $\cD_v$ with multiplicity one for $v\in 2U\oplus 2A_1(-2)$ primitive, $(v,v)=-4$ and $\div(v)=4$. 
\item $\Borch(\phi^{(1)}_{0,2A_1(2)})\in M_{10}(\Orth^+(2U\oplus 2A_1(-2)),\chi)$. Its divisor is the sum of $\cD_u$ and $\cD_v$ with multiplicity one, where $u\in 2U\oplus 2A_1(-2)$ primitive, $(u,u)=-4$ and $\div(u)=2$,  and $v\in 2U\oplus 2A_1(-2)$ primitive, $(v,v)=-4$ and $\div(v)=4$. 
\item $\Borch(\phi^{(2)}_{0,2A_1(2)})\in M_4(\Orth^+(2U\oplus 2A_1(-2)), \chi)$. Its divisor is the sum of $\cD_w$ with multiplicity one for $w\in 2U\oplus 2A_1(-2)$ primitive, $(w,w)=-8$ and $\div(w)=4$.
\item $\Phi_{12, 2A_1(2)} \in M_{12}(\Orth^+(2U\oplus 2A_1(-2)), \chi)$. Its divisor is the sum of $\cD_r$ with multiplicity one for $r\in 2U\oplus 2A_1(-2)$ with $(r,r)=-2$. 
\end{enumerate}
We also construct three additive lifts
\begin{align*}
\Grit(\eta^{12}\phi_{0,2A_1})\in M_6(\Orth^+(2U\oplus 2A_1(-2)), \chi),\\
\Grit(\eta^{12}\phi_{-2,2A_1})\in M_4(\Orth^+(2U\oplus 2A_1(-2)), \chi),\\
\Grit(\vartheta(z_1)\vartheta(z_2))\in M_1(\Orth^+(2U\oplus 2A_1(-2)), \chi).
\end{align*}

We can prove that $\Grit(\vartheta(\tau,z_1)\vartheta(\tau,z_2))=\Borch(\phi_{0,2A_1(2)})$ using the argument in \cite{GW20}. This suggests the following construction
$$
\phi_{0,2A_1(2)}= - \frac{(\vartheta(\tau,z_1)\vartheta(\tau,z_2)) \lvert  T^{(4)}_{-}(5)}{\vartheta(\tau,z_1)\vartheta(\tau,z_2)}.
$$

\begin{theorem}
Let $\Gamma_{2,4,8}(2A_1(2))$ be the subgroup of $\Orth^+(2U\oplus 2A_1(-2))$ generated by reflections associated to vectors of types $r$, $u$, $v$ and $w$ above.  Let $\Gamma_{2,4',8}(2A_1(2))$ be the subgroup of $\Orth^+(2U\oplus 2A_1(-2))$ generated by reflections associated to vectors of types $r$, $u$ and $w$ above.  
\begin{enumerate}
\item The algebra $M_*(\Gamma_{2,4,8}(2A_1(2)))$ is freely generated by $\cE_{4, 2A_1(2)}$, $\cE_{6, 2A_1(2)}$, $\Grit(\eta^{12}\phi_{0,2A_1})$, $\Grit(\eta^{12}\phi_{-2,2A_1})$ and $\Grit(\vartheta(\tau,z_1)\vartheta(\tau,z_2))^2$. The Jacobian determinant equals $$\Borch(\phi^{(1)}_{0,2A_1(2)})\Borch(\phi^{(2)}_{0,2A_1(2)})\Phi_{12, 2A_1(2)}$$ up to a constant. $(26-4=4+6+6+4+2)$
\item The algebra $M_*(\Gamma_{2,4',8}(2A_1(2)))$ is freely generated by $\cE_{4, 2A_1(2)}$, $\cE_{6, 2A_1(2)}$, $\Grit(\eta^{12}\phi_{0,2A_1})$, $\Grit(\eta^{12}\phi_{-2,2A_1})$ and $\Grit(\vartheta(\tau,z_1)\vartheta(\tau,z_2))$. The Jacobian determinant equals $$\Borch(\phi^{(1)}_{0,2A_1(2)})\Borch(\phi^{(2)}_{0,2A_1(2)})\Phi_{12, 2A_1(2)}/\Borch(\phi_{0,2A_1(2)})$$ up to a constant. $(25-4=4+6+6+4+1)$
\end{enumerate}
\end{theorem}

The formal Fourier-Jacobi expansions  for $\Gamma_{2,4,8}(2A_1(2))$ are defined as
$$
FM_{2k}(\Gamma_{2,4,8}(2A_1(2)))=\left\{ \sum_{m=0}^\infty \psi_m \xi^{m/2} \in \prod_{m=0}^\infty J_{2k,2A_1,m}^{\Orth(2A_1)}(v_\eta^{12m}) : f_m(n,\ell)=f_{2n}(m/2,\ell), \forall n,\ell,m \right\},
$$
where $J_{k,2A_1,m}^{\Orth(2A_1)}(v_\eta^{12m})$ is the space of Jacobi forms invariant under $\Orth(2A_1)$. 
\begin{corollary}
The following map is an isomorphism
\begin{align*}
M_{2k}(\Gamma_{2,4,8}(2A_1(2))) &\to FM_{2k}(\Gamma_{2,4,8}(2A_1(2))),\\
F &\mapsto \text{Fourier-Jacobi expansion of $F$}.
\end{align*}
Moreover, we have
$$
\dim M_{2k}(\Gamma_{2,4,8}(2A_1(2))) = \sum_{m=0}^\infty \dim J_{2k-6m,2A_1,m}^{w,\Orth(2A_1)}.
$$
\end{corollary}

\subsection{The \texorpdfstring{$A_2$}{A2} case}
We consider the cases of $2U\oplus A_2(-n)$ for $n=2,3$. The bigraded ring of weak Jacobi forms invariant under the Weyl group $W(A_2)$ is freely generated over $M_*(\SL_2(\ZZ))$ by three forms of index one $\phi_{0,A_2}$, $\phi_{-2,A_2}$ and $\phi_{-3,A_2}$ (see \cite{Wir92} and \cite[\S 2]{WW20}). The three unique generators (up to scale) were first constructed in \cite{Ber99}.  We fix the standard model of $A_2$:
$$
A_2= \latt{\ZZ b_1+\ZZ b_2, b_1^2=b_2^2=2, (b_1,b_2)=-1}, \quad \mathfrak{z}=z_1 b_1+z_2 b_2, \; \zeta_j=e^{2\pi i z_j}, j=1,2.
$$
For convenience, we set
\begin{align*}
P_1&=\zeta_1^{-1}+\zeta_1\zeta_2^{-1}+\zeta_2,\\
P_2&=\zeta_2^{-1}+\zeta_2\zeta_1^{-1}+\zeta_1,\\
Q&=\zeta_1^2\zeta_2^{-1}+\zeta_2^2\zeta_1^{-1}+\zeta_1\zeta_2+\zeta_1^{-2}\zeta_2+\zeta_2^{-2}\zeta_1+\zeta_1^{-1}\zeta_2^{-1}.
\end{align*}
The three generators have the following Fourier expansions:
\begin{align*}
\phi_{0,A_2}&=P_1+P_2+18+O(q)\in J_{0,A_2,1}^w,\\
\phi_{-2,A_2}&=P_1+P_2-6+O(q)\in J_{-2,A_2,1}^w,\\
\phi_{-3,A_2}&=P_1-P_3+O(q)\in J_{-3,A_2,1}^w.
\end{align*}
The orthogonal group  $\Orth(A_2)$ is generated by $W(A_2)$ and the operator $\mathfrak{z}\mapsto -\mathfrak{z}$. The forms $\phi_{0,A_2}$ and $\phi_{-2,A_2}$ are $\Orth(A_2)$-invariant. The form $\phi_{-3,A_2}$ is anti-invariant with respect to the operator  $\mathfrak{z}\mapsto -\mathfrak{z}$.

We can construct the following unique weak Jacobi forms with given Fourier expansions in terms of the above generators:
\begin{align*}
\phi_{0,A_2(2)}&=P_1+P_2+6+O(q)\in J_{0,A_2,2}^{w,\Orth(A_2)},\\
\phi^{(1)}_{0,A_2(2)}&=Q+30+O(q)\in J_{0,A_2,2}^{w,\Orth(A_2)},\\
\phi_{0,A_2(3)}&=P_1+P_2+2+O(q)\in J_{0,A_2,3}^{w,\Orth(A_2)},\\
\phi^{(1)}_{0,A_2(3)}&=Q+18+O(q)\in J_{0,A_2,3}^{w,\Orth(A_2)}.
\end{align*}

\subsubsection{The case of $n=2$}
There are three Borcherds products:
\begin{enumerate}
\item $\Borch(\phi_{0,A_2(2)})\in M_3(\Orth^+(2U\oplus A_2(-2)), \chi)$. Its divisor is the sum of $\cD_v$ with multiplicity one for $v\in 2U\oplus A_2(-2)$ primitive, $(v,v)=-12$ and $\div(v)=6$. 
\item $\Borch(\phi^{(1)}_{0,A_2(2)})\in M_{15}(\Orth^+(2U\oplus A_2(-2)),\chi)$. Its divisor is the sum of $\cD_u$ with multiplicity one for $u\in 2U\oplus A_2(-2)$ primitive, $(u,u)=-4$ and $\div(u)=2$. 
\item $\Phi_{12, A_2(2)} \in M_{12}(\Orth^+(2U\oplus A_2(-2)), \chi)$. Its divisor is the sum of $\cD_r$ with multiplicity one for $r\in 2U\oplus A_2(-2)$ with $(r,r)=-2$. 
\end{enumerate}
We also have three additive lifts
\begin{align*}
\Grit(\eta^{12}\phi_{0,A_2})\in M_6(\Orth^+(2U\oplus A_2(-2)), \chi),\\
\Grit(\eta^{12}\phi_{-2,A_2})\in M_4(\Orth^+(2U\oplus A_2(-2)), \chi),\\
\Grit(\eta^{12}\phi_{-3,A_2})\in M_3(\Orth^+(2U\oplus A_2(-2)), \chi).
\end{align*}

We remark that $\Grit(\eta^{12}\phi_{-3,A_2})=\Borch(\phi_{0,A_2(2)})$ and we can also construct $\phi_{0,A_2(2)}$ as
$$
\phi_{0,A_2(2)}=-\frac{(\eta^3(\tau)\vartheta(\tau,z_1)\vartheta(\tau,z_1-z_2)\vartheta(\tau,z_2))\lvert T_{-}^{(2)}(3)}{\eta^3(\tau)\vartheta(\tau,z_1)\vartheta(\tau,z_1-z_2)\vartheta(\tau,z_2)}.
$$

\begin{theorem}
Let $\Gamma_{2,4}(A_2(2))$ be the subgroup of $\Orth^+(2U\oplus A_2(-2))$ generated by reflections associated to vectors of types $r$ and $u$ above.  Let $\Gamma_{2,4,12}(A_2(2))$ be the subgroup of $\Orth^+(2U\oplus A_2(-2))$ generated by reflections associated to vectors of types $r$, $u$ and $v$ above.  
\begin{enumerate}
\item The graded algebra $M_*(\Gamma_{2,4}(A_2(2)))$ is freely generated by $\cE_{4, A_2(2)}$, $\cE_{6, A_2(2)}$, $\Grit(\eta^{12}\phi_{0,A_2})$, $\Grit(\eta^{12}\phi_{-2,A_2})$, and $\Grit(\eta^{12}\phi_{-3,A_2})$. The Jacobian determinant equals\\
$\Borch(\phi^{(1)}_{0,A_2(2)})\Phi_{12, A_2(2)}$ up to a constant. $(27-4=4+6+6+4+3)$
\item The graded algebra $M_*(\Gamma_{2,4,12}(A_2(2)))$ is freely generated by $\cE_{4, A_2(2)}$, $\cE_{6, A_2(2)}$, $\Grit(\eta^{12}\phi_{0,A_2})$, $\Grit(\eta^{12}\phi_{-2,A_2})$, and $\Grit(\eta^{12}\phi_{-3,A_2})^2$. The Jacobian determinant equals\\
$\Borch(\phi_{0,A_2(2)})\Borch(\phi^{(1)}_{0,A_2(2)})\Phi_{12, A_2(2)}$ up to a constant. $(30-4=4+6+6+4+6)$
\end{enumerate}
\end{theorem}

The formal Fourier-Jacobi expansions  for $\Gamma_{2,4}(A_2(2))$ are defined as
$$
FM_{k}(\Gamma_{2,4}(A_2(2)))=\left\{ \sum_{m=0}^\infty \psi_m \xi^{m/2} \in \prod_{m=0}^\infty J_{k,A_2,m}^{W(A_2)}(v_\eta^{12m}) : f_m(n,\ell)=f_{2n}(m/2,\ell), \forall n,\ell,m \right\},
$$
where $J_{k,A_2,m}^{W(A_2)}(v_\eta^{12m})$ is the space of Jacobi forms invariant under the Weyl group $W(A_2)$. 
\begin{corollary}
The following map is an isomorphism
\begin{align*}
M_{k}(\Gamma_{2,4}(A_2(2))) &\to FM_{k}(\Gamma_{2,4}(A_2(2))),\\
F &\mapsto \text{Fourier-Jacobi expansion of $F$}.
\end{align*}
Moreover, we have
$$
\dim M_{k}(\Gamma_{2,4}(A_2(2))) = \sum_{m=0}^\infty \dim J_{k-6m,A_2,m}^{w,W(A_2)}.
$$
\end{corollary}

\subsubsection{The case of $n=3$}
We need three Borcherds products and three additive lifts:
\begin{enumerate}
\item $\Borch(\phi_{0,A_2(3)})\in M_1(\Orth^+(2U\oplus A_2(-3)), \chi)$. Its divisor is the sum of $\cD_v$ with multiplicity one for $v\in 2U\oplus A_2(-3)$ primitive, $(v,v)=-18$ and $\div(v)=9$. 
\item $\Borch(\phi^{(1)}_{0,A_2(3)})\in M_{9}(\Orth^+(2U\oplus A_2(-3)),\chi)$. Its divisor is the sum of $\cD_u$ with multiplicity one for $u\in 2U\oplus A_2(-3)$ primitive, $(u,u)=-6$ and $\div(u)=3$. 
\item $\Phi_{12, A_2(3)} \in M_{12}(\Orth^+(2U\oplus A_2(-3)), \chi)$. Its divisor is the sum of $\cD_r$ with multiplicity one for $r\in 2U\oplus A_2(-3)$ with $(r,r)=-2$. 
\end{enumerate}

\begin{align*}
\Grit(\eta^{8}\phi_{0,A_2})\in M_4(\Orth^+(2U\oplus A_2(-3)), \chi),\\
\Grit(\eta^{8}\phi_{-2,A_2})\in M_2(\Orth^+(2U\oplus A_2(-3)), \chi),\\
\Grit(\eta^{8}\phi_{-3,A_2})\in M_1(\Orth^+(2U\oplus A_2(-3)), \chi).
\end{align*}

We remark that $\Grit(\eta^{8}\phi_{-3,A_2})=\Borch(\phi_{0,A_2(3)})$ and $\Grit(\Theta_{A_2(3)})=\Borch(\phi^{(1)}_{0,A_2(3)})$, where
$$
\Theta_{A_2(3)}=\eta^{15}(\tau)\vartheta(\tau,z_1+z_2)\vartheta(\tau,2z_2-z_1)\vartheta(\tau,2z_1-z_2)\in J_{9,A_2,3}.
$$
The following identities hold
$$
\phi_{0,A_2(3)}=-\frac{(\eta^{-1}(\tau)\vartheta(\tau,z_1)\vartheta(\tau,z_1-z_2)\vartheta(\tau,z_2))\lvert T_{-}^{(3)}(4)}{\eta^{-1}(\tau)\vartheta(\tau,z_1)\vartheta(\tau,z_1-z_2)\vartheta(\tau,z_2)}, \quad
\phi^{(1)}_{0,A_2(3)}=-\frac{\Theta_{A_2(3)}\lvert T_{-}^{(1)}(2)}{\Theta_{A_2(3)}}.
$$

\begin{theorem}
Let $\Gamma_{2,6}(A_2(3))$ be the subgroup of $\Orth^+(2U\oplus A_2(-3))$ generated by reflections associated to vectors of types $r$ and $u$ above.  Let $\Gamma_{2,6,18}(A_2(3))$ be the subgroup of $\Orth^+(2U\oplus A_2(-3))$ generated by reflections associated to vectors of types $r$, $u$ and $v$ above.  
\begin{enumerate}
\item The graded algebra $M_*(\Gamma_{2,6}(A_2(3)))$ is freely generated by $\cE_{4, A_2(3)}$, $\cE_{6, A_2(3)}$, $\Grit(\eta^{8}\phi_{0,A_2})$, $\Grit(\eta^{8}\phi_{-2,A_2})$, and $\Grit(\eta^{8}\phi_{-3,A_2})$. The Jacobian determinant equals\\
$\Borch(\phi^{(1)}_{0,A_2(3)})\Phi_{12, A_2(3)}$ up to a constant. $(21-4=4+6+4+2+1)$
\item The graded algebra $M_*(\Gamma_{2,6,18}(A_2(3)))$ is freely generated by $\cE_{4, A_2(3)}$, $\cE_{6, A_2(3)}$, $\Grit(\eta^{8}\phi_{0,A_2})$, $\Grit(\eta^{8}\phi_{-2,A_2})$, and $\Grit(\eta^{8}\phi_{-3,A_2})^2$. The Jacobian determinant equals\\
$\Borch(\phi_{0,A_2(3)})\Borch(\phi^{(1)}_{0,A_2(3)})\Phi_{12, A_2(3)}$ up to a constant. $(22-4=4+6+4+2+2)$
\end{enumerate}
\end{theorem}

The formal Fourier-Jacobi expansions  for $\Gamma_{2,6}(A_2(3))$ are defined as
$$
FM_{k}(\Gamma_{2,6}(A_2(3)))=\left\{ \sum_{m=0}^\infty \psi_m \xi^{m/3} \in \prod_{m=0}^\infty J_{k,A_2,m}^{W(A_2)}(v_\eta^{8m}) : f_m(n,\ell)=f_{3n}(m/3,\ell), \forall n,\ell,m \right\},
$$
where $J_{k,A_2,m}^{W(A_2)}(v_\eta^{8m})$ is the space of Jacobi forms invariant under the Weyl group $W(A_2)$. 
\begin{corollary}
The following map is an isomorphism
\begin{align*}
M_{k}(\Gamma_{2,6}(A_2(3))) &\to FM_{k}(\Gamma_{2,6}(A_2(3))),\\
F &\mapsto \text{Fourier-Jacobi expansion of $F$}.
\end{align*}
Moreover, we have
$$
\dim M_{k}(\Gamma_{2,6}(A_2(3))) = \sum_{m=0}^\infty \dim J_{k-4m,A_2,m}^{w,W(A_2)}.
$$
\end{corollary}

\subsection{The \texorpdfstring{$A_3$}{A3} case}
We study the case of $2U\oplus A_3(-2)$. It is clear that $A_3$ is isomorphic to $D_3$. We use the following model of $D_3\cong A_3$:
$$
D_3=\{ x\in \ZZ^3:  x_1+x_2+x_3 \equiv 0 \m 2\}, \; \mathfrak{z}=(z_1,z_2,z_3)\in \CC^3, \; \zeta_j=e^{2\pi i z_j}, j=1,2,3.
$$
The bigraded ring of weak Jacobi forms invariant under the Weyl group $W(A_3)$ is freely generated over $M_*(\SL_2(\ZZ))$ by four forms of index one $\phi_{0,A_3}$, $\phi_{-2,A_3}$, $\phi_{-3,A_3}$ and $\phi_{-4,A_3}$ (see \cite{Wir92} and \cite[\S 2]{WW20}). The four generators were first constructed in \cite{Ber99} and they have the following Fourier expansions
\begin{align*}
\phi_{0,A_3}&=[1,0,0] +18 +O(q) \in J_{0,A_3,1}^w,\\
\phi_{-2,A_3}&=2[\frac{1}{2},\frac{1}{2},\frac{1}{2}]_0+2[\frac{1}{2},\frac{1}{2},\frac{1}{2}]_1+[1,0,0] -22 +O(q) \in J_{-2,A_3,1}^w,\\
\phi_{-3,A_3}&=[\frac{1}{2},\frac{1}{2},\frac{1}{2}]_0-[\frac{1}{2},\frac{1}{2},\frac{1}{2}]_1 +O(q) \in J_{-3,A_3,1}^w,\\
\phi_{-4,A_3}&=[\frac{1}{2},\frac{1}{2},\frac{1}{2}]_0+[\frac{1}{2},\frac{1}{2},\frac{1}{2}]_1-[1,0,0] -2 +O(q) \in J_{-4,A_3,1}^w,
\end{align*}
where $[1,0,0]=\sum_{j=1}^3\zeta_j^{\pm 1}$,  and $[\frac{1}{2},\frac{1}{2},\frac{1}{2}]_0$ (resp. $[\frac{1}{2},\frac{1}{2},\frac{1}{2}]_1$) is the sum of terms of the form $\zeta_1^{\pm \frac{1}{2}}\zeta_2^{\pm \frac{1}{2}}\zeta_3^{\pm \frac{1}{2}}$ with even (resp. odd) number of $-\frac{1}{2}$.
 
The group  $\Orth(A_3)$ is generated by $W(A_3)$ and the operator $\mathfrak{z}\mapsto -\mathfrak{z}$. The forms of even weight are $\Orth(A_3)$-invariant. The forms of odd weight are anti-invariant with respect to the operator  $\mathfrak{z}\mapsto -\mathfrak{z}$.

We can construct the following unique weak Jacobi forms with given Fourier expansions in terms of the above generators:
\begin{align*}
\phi_{0,A_3(2)}&=[1,0,0]+6+O(q)\in J_{0,A_3,2}^{w,\Orth(A_3)},\\
\phi^{(1)}_{0,A_3(2)}&=\sum_{1\leq i< j \leq 3}\zeta_i^{\pm 1}\zeta_j^{\pm 1}+36+O(q)\in J_{0,A_3,2}^{w,\Orth(A_3)}.
\end{align*}

There are three reflective Borcherds products:
\begin{enumerate}
\item $\Borch(\phi_{0,A_3(2)})\in M_3(\Orth^+(2U\oplus A_3(-2)), \chi)$. Its divisor is the sum of $\cD_v$ with multiplicity one for $v\in 2U\oplus A_3(-2)$ primitive, $(v,v)=-8$ and $\div(v)=4$. 
\item $\Borch(\phi^{(1)}_{0,A_3(2)})\in M_{18}(\Orth^+(2U\oplus A_3(-2)),\chi)$. Its divisor is the sum of $\cD_u$ with multiplicity one for $u\in 2U\oplus A_3(-2)$ primitive, $(u,u)=-4$ and $\div(u)=2$. 
\item $\Phi_{12, A_3(2)} \in M_{12}(\Orth^+(2U\oplus A_3(-2)), \chi)$. Its divisor is the sum of $\cD_r$ with multiplicity one for $r\in 2U\oplus A_3(-2)$ with $(r,r)=-2$. 
\end{enumerate}
There are also four additive lifts:
\begin{align*}
\Grit(\eta^{12}\phi_{0,A_3})\in M_6(\Orth^+(2U\oplus A_3(-2)), \chi),\\
\Grit(\eta^{12}\phi_{-2,A_3})\in M_4(\Orth^+(2U\oplus A_3(-2)), \chi),\\
\Grit(\eta^{12}\phi_{-3,A_3})\in M_3(\Orth^+(2U\oplus A_3(-2)), \chi),\\
\Grit(\eta^{12}\phi_{-4,A_3})\in M_2(\Orth^+(2U\oplus A_3(-2)), \chi).
\end{align*}

We remark that $\Grit(\eta^{12}\phi_{-3,A_3})=\Borch(\phi_{0,A_3(2)})$ and 
$$
\phi_{0,A_3(2)}=-\frac{(\eta^{3}(\tau)\vartheta(\tau,z_1)\vartheta(\tau,z_2)\vartheta(\tau,z_3))\lvert T_{-}^{(2)}(3)}{\eta^{3}(\tau)\vartheta(\tau,z_1)\vartheta(\tau,z_2)\vartheta(\tau,z_3)}.
$$

\begin{theorem}
Let $\Gamma_{2,4}(A_3(2))$ be the subgroup of $\Orth^+(2U\oplus A_3(-2))$ generated by reflections associated to vectors of types $r$ and $u$ above.  Let $\Gamma_{2,4,8}(A_3(2))$ be the subgroup of $\Orth^+(2U\oplus A_3(-2))$ generated by reflections associated to vectors of types $r$, $u$ and $v$ above.  
\begin{enumerate}
\item The graded algebra $M_*(\Gamma_{2,4}(A_3(2)))$ is freely generated by $\cE_{4, A_3(2)}$, $\cE_{6, A_3(2)}$, $\Grit(\eta^{12}\phi_{0,A_3})$, $\Grit(\eta^{12}\phi_{-2,A_3})$, $\Grit(\eta^{12}\phi_{-3,A_3})$, and $\Grit(\eta^{12}\phi_{-4,A_3})$. The Jacobian determinant equals\\
$\Borch(\phi^{(1)}_{0,A_3(2)})\Phi_{12, A_3(2)}$ up to a constant. $(30-5=4+6+6+4+3+2)$
\item The graded algebra $M_*(\Gamma_{2,4,8}(A_3(2)))$ is freely generated by $\cE_{4, A_3(2)}$, $\cE_{6, A_3(2)}$, $\Grit(\eta^{12}\phi_{0,A_3})$, $\Grit(\eta^{12}\phi_{-2,A_3})$, $\Grit(\eta^{12}\phi_{-3,A_3})^2$, and $\Grit(\eta^{12}\phi_{-4,A_3})$. The Jacobian determinant equals\\
$\Borch(\phi_{0,A_3(2)})\Borch(\phi^{(1)}_{0,A_3(2)})\Phi_{12, A_3(2)}$ up to a constant. $(33-5=4+6+6+4+6+2)$
\end{enumerate}
\end{theorem}

The formal Fourier-Jacobi expansions  for $\Gamma_{2,4}(A_3(2))$ are defined as
$$
FM_{k}(\Gamma_{2,4}(A_3(2)))=\left\{ \sum_{m=0}^\infty \psi_m \xi^{m/2} \in \prod_{m=0}^\infty J_{k,m}^{W(A_3)}(v_\eta^{12m}) : f_m(n,\ell)=f_{2n}(m/2,\ell), \forall n,\ell,m \right\},
$$
where $J_{k,A_3,m}^{W(A_3)}(v_\eta^{12m})$ is the space of Jacobi forms invariant under the Weyl group $W(A_3)$. 
\begin{corollary}
The following map is an isomorphism
\begin{align*}
M_{k}(\Gamma_{2,4}(A_3(2))) &\to FM_{k}(\Gamma_{2,4}(A_3(2))),\\
F &\mapsto \text{Fourier-Jacobi expansion of $F$}.
\end{align*}
Moreover, we have
$$
\dim M_{k}(\Gamma_{2,4}(A_3(2))) = \sum_{m=0}^\infty \dim J_{k-6m,A_3,m}^{w,W(A_3)}.
$$
\end{corollary}

\subsection{The \texorpdfstring{$D_4$}{D4} case}
We consider the case of $2U\oplus D_4(-2)$. We fix the following model of $D_4$:
$$
D_4=\{ x\in \ZZ^4:  x_1+x_2+x_3+x_4 \equiv 0 \m 2\}, \; \mathfrak{z}=(z_1,z_2,z_3,z_4)\in \CC^4, \; \zeta_j=e^{2\pi i z_j}, j=1,2,3,4.
$$
The bigraded ring of weak Jacobi forms invariant under the Weyl group $W(D_4)$ is freely generated over $M_*(\SL_2(\ZZ))$ by four forms of index 1, i.e. $\phi_{0,D_4}$, $\phi_{-2,D_4}$, $\phi_{-4,D_4}$, $\psi_{-4,D_4}$, and one form $\phi_{-6,D_4,2}$ of index 2  (see \cite{Wir92} and \cite[\S 2]{WW20}). The five generators were constructed in \cite{AG19} and we can choose the generators of index one with the following Fourier expansions:
\begin{align*}
\phi_{0,D_4}&=[1,0,0,0] +16 +O(q) \in J_{0,D_4,1}^w,\\
\phi_{-2,D_4}&=[\frac{1}{2},\frac{1}{2},\frac{1}{2},\frac{1}{2}]_0+[\frac{1}{2},\frac{1}{2},\frac{1}{2},\frac{1}{2}]_1+[1,0,0,0] -24 +O(q) \in J_{-2,D_4,1}^w,\\
\phi_{-4,D_4}&=[\frac{1}{2},\frac{1}{2},\frac{1}{2},\frac{1}{2}]_0+[\frac{1}{2},\frac{1}{2},\frac{1}{2},\frac{1}{2}]_1-2[1,0,0,0] +O(q) \in J_{-4,D_4,1}^w,\\
\psi_{-4,D_4}&=[\frac{1}{2},\frac{1}{2},\frac{1}{2},\frac{1}{2}]_0-[\frac{1}{2},\frac{1}{2},\frac{1}{2},\frac{1}{2}]_1+O(q) \in J_{-4,D_4,1}^w,
\end{align*}
where $[1,0,0,0]=\sum_{j=1}^4\zeta_j^{\pm 1}$,  and $[\frac{1}{2},\frac{1}{2},\frac{1}{2},\frac{1}{2}]_0$ (resp. $[\frac{1}{2},\frac{1}{2},\frac{1}{2},\frac{1}{2}]_1$) is the sum of terms of the form $\zeta_1^{\pm \frac{1}{2}}\zeta_2^{\pm \frac{1}{2}}\zeta_3^{\pm \frac{1}{2}}\zeta_4^{\pm \frac{1}{2}}$ with even (resp. odd) number of $-\frac{1}{2}$.
 
The generators of type $\phi$ are invariant with respect to the operator  $z_1 \mapsto - z_1$ but the generator $\psi_{-4,D_4}$ is anti-invariant. We remark that the Weyl group $W(C_4)$ of root system $C_4$ is generated by $W(D_4)$ and the operator  $z_1 \mapsto - z_1$. Moreover, $\phi_{-2,D_4}$ is invariant under $\Orth(D_4)$. We can choose $\phi_{-6,D_4,2}$ as the generator of Weyl invariant Jacobi forms associated to root system $F_4$ (see \cite{Wir92}). Since $W(F_4)=\Orth(D_4)$,  we can assume that $\phi_{-6,D_4,2}$ is invariant under $\Orth(D_4)$.

We have the following unique weak Jacobi forms with given Fourier expansions, which can be constructed in terms of the above generators.
\begin{align*}
\phi_{0,D_4(2)}&=[1,0,0,0]+4+O(q)\in J_{0,D_4,2}^{w,W(C_4)},\\
\varphi_{0,D_4(2)}&=[1,0,0,0]+[\frac{1}{2},\frac{1}{2},\frac{1}{2},\frac{1}{2}]_0+[\frac{1}{2},\frac{1}{2},\frac{1}{2},\frac{1}{2}]_1+12+O(q)\in J_{0,D_4,2}^{w,\Orth(D_4)},\\
\phi^{(1)}_{0,D_4(2)}&=\sum_{1\leq i< j \leq 4}\zeta_i^{\pm 1}\zeta_j^{\pm 1}+48+O(q)\in J_{0,D_4,2}^{w,\Orth(D_4)}.
\end{align*}

We need the following four reflective Borcherds products:
\begin{enumerate}
\item $\Borch(\phi_{0,D_4(2)})\in M_2(\Gamma_1, \chi)$. Its divisor is the sum of $\cD_{v_1}$ with multiplicity one for vectors $v_1\in 2U\oplus D_4(-2)$ primitive, $(v_1,v_1)=-8$, $\div(v_1)=4$, and $\frac{v_1}{2}-(1,0,0,0)\in 2U\oplus D_4(-1)$. 
\item $\Borch(\varphi_{0,D_4(2)})\in M_6(\Orth^+(2U\oplus D_4(-2)), \chi)$. Its divisor is the sum of $\cD_{v}$ with multiplicity one for vectors $v\in 2U\oplus D_4(-2)$ primitive, $(v,v)=-8$, and $\div(v)=4$. 
\item $\Borch(\phi^{(1)}_{0,D_4(2)})\in M_{24}(\Orth^+(2U\oplus D_4(-2)),\chi)$. Its divisor is the sum of $\cD_u$ with multiplicity one for $u\in 2U\oplus D_4(-2)$ primitive, $(u,u)=-4$, and $\div(u)=2$. 
\item $\Phi_{12, D_4(2)} \in M_{12}(\Orth^+(2U\oplus D_4(-2)), \chi)$. Its divisor is the sum of $\cD_r$ with multiplicity one for $r\in 2U\oplus D_4(-2)$ with $(r,r)=-2$. 
\end{enumerate}
Here, $\Gamma_1$ is the subgroup generated by $\widetilde{\Orth}^+(2U\oplus D_4(-2))$ and $W(C_4)$.
If the input is invariant up to a character with respect to $W(C_4)$, then the additive lift is a modular form for $\Gamma_1$.
Thus we have the following additive lifts:
\begin{align*}
\Grit(\eta^{12}\phi_{0,D_4})&\in M_6(\Gamma_1, \chi),\\
\Grit(\eta^{12}\phi_{-2,D_4})&\in M_4(\Orth^+(2U\oplus D_4(-2)), \chi),\\
\Grit(\eta^{12}\phi_{-4,D_4})&\in M_2(\Gamma_1, \chi),\\
\Grit(\eta^{12}\psi_{-4,D_4})&\in M_2(\Gamma_1, \chi),\\
\Grit(\eta^{24}\phi_{-6,D_4,2})&\in M_6(\Orth^+(2U\oplus D_4(-2)), \chi).
\end{align*}

We remark that $\Grit(\eta^{12}\psi_{-4,D_4})=\Borch(\phi_{0,D_4(2)})$ and 
$$
\phi_{0,D_4(2)}=-\frac{(\vartheta(\tau,z_1)\vartheta(\tau,z_2)\vartheta(\tau,z_3)\vartheta(\tau,z_4))\lvert T_{-}^{(2)}(3)}{\vartheta(\tau,z_1)\vartheta(\tau,z_2)\vartheta(\tau,z_3)\vartheta(\tau,z_4)}.
$$

\begin{theorem}
Let $\Gamma_{2,4}(D_4(2))$ be the subgroup of $\Orth^+(2U\oplus D_4(-2))$ generated by reflections associated to vectors of types $r$ and $u$ above.  Let $\Gamma_{2,4,8'}(D_4(2))$ be the subgroup of $\Orth^+(2U\oplus D_4(-2))$ generated by reflections associated to vectors of types $r$, $u$ and $v_1$ above.   
\begin{enumerate}
\item The graded algebra $M_*(\Gamma_{2,4}(D_4(2)))$ is freely generated by $\cE_{4, D_4(2)}$, $\cE_{6, D_4(2)}$, $\Grit(\eta^{12}\phi_{0,D_4})$, $\Grit(\eta^{12}\phi_{-2,D_4})$, $\Grit(\eta^{12}\phi_{-4,D_4})$, $\Grit(\eta^{12}\psi_{-4,D_4})$, and $\Grit(\eta^{24}\phi_{-6,D_4,2})$. The Jacobian determinant equals
$\Borch(\phi^{(1)}_{0,D_4(2)})\Phi_{12, D_4(2)}$ up to a constant. $(36-6=4+6+6+4+2+2+6)$
\item The graded algebra $M_*(\Gamma_{2,4,8'}(D_4(2)))$ is freely generated by $\cE_{4, D_4(2)}$, $\cE_{6, D_4(2)}$, $\Grit(\eta^{12}\phi_{0,D_4})$, $\Grit(\eta^{12}\phi_{-2,D_4})$, $\Grit(\eta^{12}\phi_{-4,D_4})$, $\Grit(\eta^{12}\psi_{-4,D_4})^2$, and $\Grit(\eta^{24}\phi_{-6,D_4,2})$. The Jacobian determinant equals
$\Borch(\phi_{0,D_4(2)})\Borch(\phi^{(1)}_{0,D_4(2)})\Phi_{12, D_4(2)}$ up to a constant. $(38-6=4+6+6+4+2+4+6)$
\end{enumerate}
\end{theorem}

The formal Fourier-Jacobi expansions  for $\Gamma_{2,4}(D_4(2))$ are defined as
$$
FM_{2k}(\Gamma_{2,4}(D_4(2)))=\left\{ \sum_{m=0}^\infty \psi_m \xi^{m/2} \in \prod_{m=0}^\infty J_{2k,m}^{W(D_4)}(v_\eta^{12m}) : f_m(n,\ell)=f_{2n}(m/2,\ell), \forall n,\ell,m \right\},
$$
where $J_{k,D_4,m}^{W(D_4)}(v_\eta^{12m})$ is the space of Jacobi forms invariant under the Weyl group $W(D_4)$. 
\begin{corollary}
The following map is an isomorphism
\begin{align*}
M_{2k}(\Gamma_{2,4}(D_4(2))) &\to FM_{2k}(\Gamma_{2,4}(D_4(2))),\\
F &\mapsto \text{Fourier-Jacobi expansion of $F$}.
\end{align*}
Moreover, we have
$$
\dim M_{2k}(\Gamma_{2,4}(D_4(2))) = \sum_{m=0}^\infty \dim J_{2k-6m,D_4,m}^{w,W(D_4)}.
$$
\end{corollary}

The above corollary implies that for an arbitrary $\phi\in J_{k,D_4,m}^{w,W(D_4)}$ there is a modular form of weight $k+6m$ on $\Gamma_{2,4}(D_4(2))$ whose first non-vanishing Fourier-Jacobi coefficient is $\eta^{12m}\phi\cdot \xi^{m/2}$. Let $\phi_{0,F_4,1}$, $\phi_{-2,F_4,1}$, $\phi_{-6,F_4,2}$, $\phi_{-8,F_4,2}$ and $\phi_{-12,F_4,3}$ be the generators of the bigraded ring of $W(F_4)$-invariant weak Jacobi forms (see \cite{Wir92}). Then there exist modular forms of weights $6$, $4$, $6$, $4$, $6$ on $\Gamma_{2,4}(D_4(2))$ whose first non-vanishing Fourier-Jacobi coefficients are respectively $\eta^{12}\phi_{0,F_4,1}\cdot \xi^{1/2}$, $\eta^{12}\phi_{-2,F_4,1}\cdot \xi^{1/2}$, $\eta^{24}\phi_{-6,F_4,2}\cdot \xi$, $\eta^{24}\phi_{-8,F_4,2}\cdot \xi$ and $\eta^{36}\phi_{-12,F_4,3}\cdot \xi^{3/2}$. Let $\Gamma_{2,4,8}(D_4(2))$ be the subgroup of $\Orth^+(2U\oplus D_4(-2))$ generated by reflections associated to vectors of types $r$, $u$ and $v$ above. We can assume that the five functions are modular forms with trivial character for $\Gamma_{2,4,8}(D_4(2))$ because we can consider their invariants under $\Gamma_{2,4,8}(D_4(2))/\Gamma_{2,4}(D_4(2))$. The Jacobian of the five forms and $\cE_{4, D_4(2)}$ and $\cE_{6, D_4(2)}$ is not identically zero and is a modular form of weight 42 on $\Gamma_{2,4,8}(D_4(2))$ with the determinant character. Thus the Jacobian is equal to $\Borch(\varphi_{0,D_4(2)})\Borch(\phi^{(1)}_{0,D_4(2)})\Phi_{12, D_4(2)}$ up to a constant multiple. We then prove the following theorem.

\begin{theorem}\label{th:D4F4}
The graded algebra $M_*(\Gamma_{2,4,8}(D_4(2)))$ is freely generated by three forms of weight $4$ and four forms of weight $6$. 
\end{theorem}

\begin{remark}
In fact, the eight groups $\Gamma_{2,4}(A_1(2))$, $\Gamma_{2,6}(A_1(3))$, $\Gamma_{2,8}(A_1(4))$, $\Gamma_{2,4,8}(2A_1(2))$, $\Gamma_{2,4,12}(A_2(2))$, $\Gamma_{2,6,18}(A_2(3))$, $\Gamma_{2,4,8}(A_3(2))$ and $\Gamma_{2,4,8}(D_4(2))$ are the maximal reflection subgroups contained in the corresponding integral orthogonal groups. In other words, each of them is generated by all reflections in the integral orthogonal group.
\end{remark}

\begin{remark}
In \cite{WW20}, B. Williams and the author proved that for some root system $R$ the algebra of modular forms for the group generated by the discriminant kernel of $2U\oplus L_R(-1)$ and the Weyl group $W(R)$ is a free algebra generated by $\rank(R)+3$ forms of weights $4$, $6$ and $-k_j+12m_j$, where $L_R$ is the root lattice generated by $R$, and the pairs $(-k_j,m_j)$ are the weights and indices of Weyl invariant weak Jacobi forms of type $R$. In this paper, we in fact construct free algebras of modular forms for rescaled root lattices, i.e. $2U\oplus L_R(-t)$. Unlike the previous cases, these reflection groups $\Gamma$ in the present paper do not contain the discriminant kernel by the classification result \cite[Theorem 4.4]{Wan20}.  That is why we use additive lifts with characters to construct generators. We next explain the relation between the weights of generators of modular forms and the weights of generators of Jacobi forms. Let $\Gamma$ be one of the 16 reflection groups. Let $W(R)$ be the maximal Weyl group contained in $\Gamma$. Then the weights of generators of $M_*(\Gamma)$ is given by $4$, $6$ and $-k_j+12m_j/t$, where the lattice model of $\Gamma$ is $2U\oplus L_R(-t)$, and the pairs $(-k_j,m_j)$ are the weights and indices of generators of $W(R)$-invariant weak Jacobi forms. For example, the lattice model of  $\Gamma_{2,4,8}(D_4(2))$ is $2U\oplus D_4(-2)$ and the maximal Weyl group contained in $\Gamma_{2,4,8}(D_4(2))$ is $W(F_4)$. The weights and indices of $W(F_4)$-invariant weak Jacobi forms are respectively $(0,1)$, $(-2,1)$, $(-6,2)$, $(-8,2)$ and $(-12,3)$. Thus the generators of $M_*(\Gamma_{2,4,8}(D_4(2)))$ have weights $4$, $6$, $6$, $4$, $6$, $4$ and $6$, which coincides with Theorem \ref{th:D4F4}.
\end{remark}

\bigskip

\noindent
\textbf{Acknowledgements} 
The author is grateful to Max Planck Institute for Mathematics in Bonn for its hospitality and financial support. 

\bibliographystyle{amsalpha}

\end{document}